\documentclass[12pt,twoside]{article}

\usepackage{amsmath,amssymb,amsthm,amscd,amsfonts,graphicx,color,accents}%
\usepackage{url,verbatim}
\usepackage[colorlinks=true,citecolor=blue,urlcolor=blue,linkcolor=blue]{hyperref}

\DeclareMathOperator{\diag}{diag}
\DeclareMathOperator{\tr}{tr}
\DeclareMathOperator{\Cof}{Cof} 

\numberwithin{equation}{section}

\theoremstyle{plain} 
\newtheorem{theorem}{Theorem}[section]

\newtheorem{lemma}[theorem]{Lemma}

\theoremstyle{definition}

\newtheorem{remark}[theorem]{Remark}

  \renewenvironment{thebibliography}[1]{%
   \begin{oldthebibliography}{#1}%
    \setlength{\parskip}{0.2ex}%
     \setlength{\itemsep}{0.0in}%
  }%
  {%
    \end{oldthebibliography}%
  }

\begin{document}

\title{\bf\large The Multiple Roots Phenomenon in Maximum Likelihood Estimation for Factor Analysis}

\author{
Elizabeth Gross\footnote{Department of Mathematics, San Jos\'e State University, \textit{elizabeth.gross@sjsu.edu}}, 
Sonja Petrovi\'c\thanks{Department of Applied Mathematics, Illinois Institute of Technology, \textit{sonja.petrovic@iit.edu}}, 
Donald Richards\thanks{Department of Statistics, Pennsylvania State University, \textit{richards@stat.psu.edu}}, 
Despina Stasi\thanks{Department of Applied Mathematics, Illinois Institute of Technology, \textit{stasdes@iit.edu}}
}
\bigskip


\maketitle

\begin{abstract}
Multiple root estimation problems in statistical inference arise in many contexts in the literature.  In the context of maximum likelihood estimation, the existence of multiple roots causes uncertainty in the computation of maximum likelihood estimators using hill-climbing algorithms, and consequent difficulties in the resulting statistical inference.  

In this paper, we study the multiple roots phenomenon in maximum likelihood estimation for factor analysis.  We prove that the corresponding likelihood equations have uncountably many feasible solutions even in the simplest cases.  For the case in which the observed data are two-dimensional and the unobserved factor scores are one-dimensional, we prove that the solutions to the likelihood equations form a one-dimensional real curve. 
\end{abstract}

\section{Introduction}
\label{introduction}

Systems of equations which have multiple roots, such as critical equations in estimation problems, arise in many contexts in statistical inference and there  is now a large literature on these problems. 
We refer to Small, et al.~\cite{Smalletal} for an
extensive
 account of the history and methods for investigating numerous multiple roots problems; Buot, et al.~\cite{BuotRichards,Buotetal} for enumeration of the roots of likelihood equations arising from mixture models or from the Behrens-Fisher problem;  Catanese, et al.~\cite{Cataneseetal} for algebraic aspects of some of these problems and Gross, et al.~\cite{Grossetal} for applications to variance component models; Hoferkamp and Peddada \cite{HoferkampPeddada} for order-restricted inference; and Zwiernik, et al.~\cite{Zwierniketal} for aspects of linear Gaussian covariance models. 

In the context of estimation by the method of maximum likelihood, the existence of multiple roots for the likelihood equations causes uncertainty in the computation of maximum likelihood estimators using hill-climbing algorithms.  Consequently, the resulting statistical inference faces the deeper difficulty that it is necessary to determine which of the multiple roots produces the maximum value of the likelihood function.  

It is well-known that maximum likelihood estimators (MLEs) enjoy many appealing asymptotic properties.  Subject to various regularity conditions and as the sample size tends to infinity, the MLE of a parameter converges in probability to the true value of the parameter, a property known as {\it asymptotic consistency}.  MLEs also are {\it asymptotically efficient}: The variance of the MLE converges to the minimal value among all consistent estimators.  Moreover, the asymptotic distribution of MLEs are known, and this enables standard statistical procedures, such as the construction of confidence regions and tests of hypotheses.  

Despite the large-sample advantages of MLEs, there are instances in which large samples are unavailable, e.g., in the statistical analysis of clinical trials data.  On the other hand, from a purely inferential perspective, one must bear in mind that, as the sample size increases, the likelihood of rejecting the null hypothesis increases simply as an artifact of said increasing sample size.  Consequently, it is important to investigate maximum likelihood estimation problems in which sample sizes are small.  

In this paper, we study the multiple roots phenomenon in maximum likelihood estimation for the problem of factor analysis; cf., Anderson \cite{Anderson}, Harman \cite{Harman}.  The factor analysis problem is old and venerable, with many applications to the social sciences, and its difficult mathematical nature is well-known; cf., Adachi \cite{Adachi}, Anderson and Rubin \cite{AndersonRubin}, and Rubin and Thayer \cite{RubinThayer}.  Starting with the formulation of the factor analysis model in \cite{RubinThayer}, we prove that the corresponding likelihood equations have uncountably many feasible solutions even in the simplest cases.  For the case in which the observed data are two-dimensional and the unobserved factor scores are one-dimensional, we prove that the solutions to the likelihood equations form a one-dimensional {\it real} curve.  


\section{A model for factor analysis}
\label{factoranalysis}

Rubin and Thayer \cite{RubinThayer} considered a problem in factor analysis, as follows.  Let $Y$ denote a $n \times p$ matrix of observed data. Let $Z$ be a $n \times q$ matrix of unobserved factor scores, where $q < p$, and each row of $Z$ is assumed to be marginally distributed as $N_q(0,R)$, a $q$-dimensional multivariate normal distribution with mean vector $0$ and $q \times q$ {\it correlation} matrix $R = (r_{ij})$, i.e., $r_{jj} = 1$ for all $j=1,\ldots,q$.  It also is assumed that the rows of the matrix $(Y,Z)$ are mutually independent and identically distributed.  

For $i=1,\ldots,n$, denote by $Y_i$ and $Z_i$ the $i$th rows of $Y$ and $Z$, respectively.  We suppose that the conditional distribution of $Y_i$, given $Z$ is given by \mbox{$Y_i|Z \sim N_p(\alpha+Z_i\beta,\tau^2)$,} where $\tau^2 = \diag(\tau_1^2,\ldots,\tau_p^2)$ is a diagonal matrix and $\beta = (\beta_{ij})$ is a $q \times p$ matrix called the {\it factor-loading matrix}.  By centering the matrix $Y$ it can be assumed, with no loss of generality that $\alpha = 0$.  Then the problem of factor analysis is to calculate the maximum likelihood estimators of the parameters $\beta$, $\tau^2$, and $R$.  

It was noted in \cite[p.~75]{RubinThayer} that the likelihood function for this factor analysis model generally has multiple local maxima.  The issue of the number of critical points, that is the number of roots of the likelihood equations,  was left open, however, and this is the focus of the present paper.

\section{The derivatives of the likelihood function}
\label{derivatives}

Following \cite{RubinThayer}, we write the log-likelihood function in the form 
$$
\ell(\tau^2,\beta,R) = -\tfrac12 n \, f(\tau^2,\beta,R)
$$
with $C_{yy} = (c_{ij})$ denoting the sample covariance matrix constructed from $Y$, and
\begin{equation}
\label{eq: loglikelihood}
f(\tau^2,\beta,R) = \log \det(\tau^2 + \beta'R\beta) + \tr C_{yy}(\tau^2 + \beta'R\beta)^{-1}.
\end{equation}
Thus, we wish to {\it minimize} the function $f$ with respect to the parameters $(\beta,\tau^2,R)$.  

To calculate the derivatives of $f$ with respect to $\tau$ and $\beta$, we will apply repeatedly the following result provided by Magnus and Neudecker \cite[p. 169 (Theorem 1) and p. 171 (Theorem 3)]{MagnusNeudecker}.  

\begin{lemma}
Suppose that the square matrix $A(t)$ is a differentiable function of a real parameter $t$.  Then 
\begin{equation}
\label{derivativelogdet}
\frac{\partial}{\partial t} \log\det(A(t)) = \tr \Big[A(t)^{-1} \frac{\partial}{\partial t} A(t)\Big]
\end{equation}
and 
\begin{equation}
\label{derivativeinvdet}
\frac{\partial}{\partial t} A(t)^{-1} = -A(t)^{-1}\Big[\frac{\partial}{\partial t} A(t)\Big]A(t)^{-1}.
\end{equation}
\end{lemma}

\bigskip

Denote by $E_{kk}$ the $p \times p$ matrix with entry $1$ in the $(k,k)$th position and zeros elsewhere.  It is straightforward to verify that 
$$
\frac{\partial}{\partial \tau_k} (\tau^2 + \beta'R\beta) = 2\tau_k E_{kk},
$$
and then it follows from formula \eqref{derivativelogdet}  
that 
\begin{align*}
\frac{\partial}{\partial \tau_k} \log\det(\tau^2 + \beta'R\beta) &= 2\tau_k\tr[(\tau^2 + \beta'R\beta)^{-1}E_{kk}] \\
&= 2\tau_k \det(\tau^2+\beta'R\beta)^{-1} \, \Cof_{kk}(\tau^2+\beta'R\beta),
\end{align*}
where $\Cof_{ij}(A)$ is the $(i,j)$th cofactor of the matrix $A$.  

Next, we apply the formula (\ref{derivativeinvdet}) to obtain 
\begin{align*}
\frac{\partial}{\partial \tau_k} \tr C_{yy}(\tau^2 + \beta'R\beta)^{-1} &= \tr [C_{yy} \frac{\partial}{\partial \tau_k} (\tau^2 + \beta'R\beta)^{-1} \\
&= -2\tau_k \tr C_{yy}(\tau^2 + \beta'R\beta)^{-1} E_{kk} (\tau^2 + \beta'R\beta)^{-1}.
\end{align*}
This latter expression seems formidable initially, but it can be rewritten in terms of the entries of $C_{yy}$ and the cofactors of $\tau^2+\beta'R\beta$, as follows:  Since 
$$
(\tau^2 + \beta'R\beta)^{-1} = \det(\tau^2 + \beta'R\beta)^{-1} \Big(\Cof_{ij}(\tau^2 + \beta'R\beta)\Big)
$$
then 
\begin{align*}
(\tau^2 + \beta'R&\beta)^{-1}E_{kk}(\tau^2 + \beta'R\beta)^{-1} \\
&= \det(\tau^2 + \beta'R\beta)^{-2} \Big(\Cof_{ij}(\tau^2 + \beta'R\beta)\Big)E_{kk}\Big(\Cof_{ij}(\tau^2 + \beta'R\beta)\Big) \\
&= \det(\tau^2 + \beta'R\beta)^{-2} \Big(\Cof_{ik}(\tau^2 + \beta'R\beta) \cdot \Cof_{kj}(\tau^2 + \beta'R\beta)\Big).
\end{align*}
Therefore, 
\begin{multline*}
\tr C_{yy}(\tau^2 + \beta'R\beta)^{-1} E_{kk} (\tau^2 + \beta'R\beta)^{-1} \\
= \det(\tau^2 + \beta'R\beta)^{-2} \sum_{i,j=1}^p c_{ij} \Cof_{jk}(\tau^2 + \beta'R\beta) \cdot \Cof_{ki}(\tau^2 + \beta'R\beta).
\end{multline*}
Collecting together the formulas above, we obtain 
\begin{align*}
\frac{\partial}{\partial \tau_k} f(\tau^2,\beta,R) &= 2\tau_k \det(\tau^2+\beta'R\beta)^{-1} \, \Cof_{kk}(\tau^2+\beta'R\beta) \\
& \quad 
-2\tau_k \det(\tau^2 + \beta'R\beta)^{-2} \sum_{i,j=1}^p c_{ij} \Cof_{jk}(\tau^2 + \beta'R\beta) \cdot \Cof_{ki}(\tau^2 + \beta'R\beta).
\end{align*}

Next, we calculate the derivative of $f(\tau^2,\beta,R)$ with respect to $\beta$.  Let $E_{kl}$ be a $q \times p$ matrix with entry $1$ in the $(k,l)$th position and zeros elsewhere; that is, the $(i,j)$th entry of $E_{kl}$ is $\delta_{ik}\delta_{jl}$, where $\delta_{ij}$ denotes Kronecker's delta.  Then 
\begin{align*}
\frac{\partial}{\partial\beta_{kl}} (\beta'R\beta) &= (\frac{\partial}{\partial\beta_{kl}} \beta)'R\beta + \beta'R\frac{\partial}{\partial\beta_{kl}} \beta \\
&= E_{kl}'R\beta + \beta'RE_{kl}.
\end{align*}
Since 
\begin{align*}
E_{kl}'R\beta = (\delta_{ik}\delta_{jl})'R\beta &= (\delta_{il}\delta_{jk})R\beta \\
&= \Big(\sum_{m=1}^q \delta_{il}\delta_{mk}(R\beta)_{mj}\Big) = \big(\delta_{il}(R\beta)_{kj}\big),
\end{align*}
then 
$$
\beta'RE_{kl} = (E_{kl}'R\beta)' = \big(\delta_{jl}(R\beta)_{ki}\big),
$$
and hence 
$$
\frac{\partial}{\partial\beta_{kl}} (\beta'R\beta) = \big(\delta_{il}(R\beta)_{kj} + \delta_{jl}(R\beta)_{ki}\big).
$$
Therefore, by formula (\ref{derivativelogdet}), 
\begin{align*}
\frac{\partial}{\partial\beta_{kl}} \log&\det(\tau^2 + \beta'R\beta) \\
&= \tr \Big[(\tau^2 + \beta'R\beta)^{-1} \frac{\partial}{\partial\beta_{kl}} (\tau^2 + \beta'R\beta)\Big] \\
&= \tr \Big[(\tau^2 + \beta'R\beta)^{-1} \big(\delta_{il}(R\beta)_{kj} + \delta_{jl}(R\beta)_{ki}\big) \Big] \\
&= \det(\tau^2 + \beta'R\beta)^{-1} \cdot \sum_{i,j=1}^p \Cof_{ij}(\tau^2 + \beta'R\beta) \cdot (\delta_{jl}(R\beta)_{ki} + \delta_{il}(R\beta)_{kj}).
\end{align*}

Next, by formula (\ref{derivativeinvdet}), 
\begin{align*}
\frac{\partial}{\partial\beta_{kl}} (\tau^2 + \beta'R\beta)^{-1} &= -(\tau^2 + \beta'R\beta)^{-1} \Big[\frac{\partial}{\partial\beta_{kl}}(\tau^2 + \beta'R\beta)\Big] (\tau^2 + \beta'R\beta)^{-1} \\
&= -(\tau^2 + \beta'R\beta)^{-1} \big(\delta_{il}(R\beta)_{kj} + \delta_{jl}(R\beta)_{ki}\big) (\tau^2 + \beta'R\beta)^{-1},
\end{align*}
Therefore, 
$$
\frac{\partial}{\partial\beta_{kl}} \tr C_{yy}(\tau^2 + \beta'R\beta)^{-1} = -\tr C_{yy}(\tau^2 + \beta'R\beta)^{-1} \big(\delta_{il}(R\beta)_{kj} + \delta_{jl}(R\beta)_{ki}\big) (\tau^2 + \beta'R\beta)^{-1}.
$$
Consequently, 
\begin{align*}
\frac{\partial}{\partial\beta_{kl}} f(\tau^2,\beta,R) &= \det(\tau^2 + \beta'R\beta)^{-1} \cdot \sum_{i,j=1}^p \Cof_{ij}(\tau^2 + \beta'R\beta) \cdot (\delta_{jl}(R\beta)_{ki} + \delta_{il}(R\beta)_{kj}) \\
& \qquad
-\tr C_{yy}(\tau^2 + \beta'R\beta)^{-1} \big(\delta_{il}(R\beta)_{kj} + \delta_{jl}(R\beta)_{ki}\big) (\tau^2 + \beta'R\beta)^{-1}.
\end{align*}

Finally, we calculate the derivatives with respect to the parameters $r_{kl}$, $k < l$.  On writing 
$$
\beta'R\beta = \big(\sum_{m,n=1}^q \beta_{mi}r_{mn}\beta_{nj}\big) = \big(\sum_{m=1}^q \beta_{mi}\beta_{mj} + 2\sum_{1\le m<n\le q} r_{mn}\beta_{mi}\beta_{nj}\big),
$$
we obtain 
$$
\frac{\partial}{\partial r_{kl}} \beta'R\beta = 2r_{kl}\big(\beta_{ki}\beta_{lj}\big)
$$
where, for fixed $(k,l)$, $(\beta_{ki}\beta_{lj})$ denotes the matrix with generic $(i,j)$th entry $\beta_{ki}\beta_{lj}$.  Consequently, 
\begin{align*}
\frac{\partial}{\partial r_{kl}} \log\det(\tau^2 + \beta'R\beta) &= \tr\big[(\tau^2 + \beta'R\beta)^{-1}\frac{\partial}{\partial r_{kl}}(\tau^2 + \beta'R\beta)\big] \\
&= 2r_{kl}\tr\big[(\tau^2 + \beta'R\beta)^{-1}\big(\beta_{ki}\beta_{lj}\big)\big].
\end{align*}

Next, 
\begin{align*}
\frac{\partial}{\partial r_{kl}} (\tau^2 + \beta'R\beta)^{-1} &= -(\tau^2 + \beta'R\beta)^{-1} \Big[\frac{\partial}{\partial r_{kl}}(\tau^2 + \beta'R\beta)\Big](\tau^2 + \beta'R\beta)^{-1} \\
&= -2r_{kl}(\tau^2 + \beta'R\beta)^{-1} \big(\beta_{ki}\beta_{lj}\big) (\tau^2 + \beta'R\beta)^{-1},
\end{align*}
and therefore 
\begin{align*}
\frac{\partial}{\partial r_{kl}} f(\tau^2,\beta,R) &= 2r_{kl}\tr\big[(\tau^2 + \beta'R\beta)^{-1}\big(\beta_{ki}\beta_{lj}\big)\big] \\
&\qquad 
-2r_{kl}\tr C_{yy}(\tau^2 + \beta'R\beta)^{-1} \big(\beta_{ki}\beta_{lj}\big) (\tau^2 + \beta'R\beta)^{-1} \\
&= 2r_{kl} \tr \Big[(\tau^2 + \beta'R\beta)^{-1}\big(\beta_{ki}\beta_{lj}\big)\big(I_p
 -(\tau^2 + \beta'R\beta)^{-1}C_{yy}\big)\Big].
\end{align*}

\section{The solutions of the likelihood equations}
\label{solutionset}

The likelihood equations derived in the previous section are rational equations whose denominators are $(\det(\tau^2 + \beta'R\beta))^2$.  In order to work with polynomial equations, we will substitute $\gamma$ for $(\det(\tau^2 + \beta'R\beta))^{-1}$ in the likelihood equations, add to the system of equations the new equation 
$$
\gamma \det(\tau^2 + \beta'R\beta)-1=0,
$$
and denote by $\mathbf f_{p,q,C_{yy}}=0$ the resulting expanded system of equations.  

The set $V_{p,q,C_{yy}} \subset \mathbb R^{p + pq + q(q-1)/2 + 1}$ of real solutions to the system $\mathbf f_{p,q, C_{yy}} = 0$ is equal to the set 
\begin{equation}
\begin{split}
V_{p,q, C_{yy}} = \bigg\{(\tau, \beta, R, \gamma) \ : \ (\tau, \beta, R) \text{ is a critical point of } f(\tau^2, \beta, R),  \\ \det(\tau^2 + \beta'R\beta) \neq 0, \text{ and } \gamma=\frac{1}{\det(\tau^2 + \beta'R\beta)}\bigg\}.
\end{split}
\end{equation}

We now set $p=2$ and $q=1$.  Let us consider the polynomials in $\mathbf f_{2,1, C_{yy}}$ and treat the entries of $C_{yy} = (c_{ij})$ as unknown variables.  Furthermore, for simplification, we make the substitution, 
\begin{equation}
\label{eq:xvariables}
\left(\begin{array}{cc}x_{11} & x_{12} \\x_{12} & x_{22}\end{array}\right) = \tau^2 + \beta'R\beta,
\end{equation} 
and add to $\mathbf f_{2,1,C_{yy}}$ the equations from \eqref{eq:xvariables}.  This new set of polynomials defines an ideal $I_{2,1} \subseteq \mathbb Q[ c_{11}, c_{12}, c_{22}, \tau_1, \tau_2, \beta_{11}, \beta_{12}, \gamma, x_{11}, x_{12}, x_{22}]$.  
As we are interested in the solution set associated with these polynomials,  we study the variety $V(I_{2,1}) \subseteq \mathbb C^{11}$ defined to the set of all points $p=(c_{11}, c_{12}, c_{22}, \tau_1, \tau_2, \beta_{11}, \beta_{12}, \gamma, x_{11}, x_{12}, x_{22}) \in \mathbb C^{11}$ such that 
$g(p) = 0$ for all polynomials $g \in I_{2,1}$.
To compute the dimension of the set $V(I_{2,1})$, we first decompose it into a union of irreducible varieties. The corresponding algebraic object is the intersection of the primary ideals. For an introduction to ideals, varieties, and primary decompositions see \cite{Coxetal}.

Using  {\tt Macaulay2} (Grayson and Stillman \cite{GraysonStillman}), we obtain the decomposition $I_{2,1}=J_1 \cap J_2$ of $I_{2,1}$ into an intersection of primary ideals $J_1$ and $J_2$, where 
\begin{align*}
J_1 &= \langle x_{22}-c_{22}, x_{12}-c_{12}, x_{11}-c_{11}, \beta_{11}\beta_{12}-c_{12}, \tau_2^2+\beta_{12}^2-c_{22}, \tau_1^2+\beta_{11}^2-c_{11}, \\
& \qquad\qquad \gamma c_{12}^2-\gamma c_{11}c_{22}+1 \rangle \\
\intertext{and} 
J_2 &= \langle x_{22}-c_{22}, x_{12}, x_{11}-c_{11}, \beta_{12}, \beta_{11}, \tau_2^2-c_{22}, \tau_1^2-c_{11}, \gamma c_{11}c_{22}-1\rangle.
\end{align*}
Moreover, we obtain from {\tt Macaulay2} that the dimensions of the components $V(J_1)$  and $V(J_2)$ corresponding to the ideals $J_1$ and $J_2$ are 4 and 3, respectively. These computations can be confirmed using numerical algebraic geometry software such as  {\tt PHCpack} (Verschelde \cite{Verschelde}).
  Since there are no relations between the variables $c_{11}$, $c_{12}$, and $c_{22}$, in other words, they are free variables, it follows that if we choose $C_{yy}$ generically from the set of $p \times p$ positive definite matrices then the solution set $V_{2,1,C_{yy}}$ contains a curve (i.e., a one-dimensional component) and a set of points (i.e., a zero-dimensional component).

Moreover, the curve in $V_{2,1,C_{yy}}$ can be described parametrically in terms of $t$ for all values of $t$ such that $c_{12}^2/c_{22} \leq  t^2 \leq c_{11}$: 
\begin{alignat}{4}
  &\beta_{11} &&= t             \qquad\qquad && \beta_{12} &&= \frac{c_{12}}{t} \label{eq:onedim1} \\
  &\tau_1^2   &&= c_{11} - t^2  \qquad\qquad && \tau_2^2 &&= c_{22} - \frac{c_{12}^2}{t^2} \label{eq:onedim2}
\end{alignat}
Since $C_{yy}$ is positive definite then the interval $\{t \in \mathbb{R}: \ c_{12}^2/c_{22} \leq t^2 \leq c_{11}\}$ has positive measure.

The isolated points in $V_{2,1,C_{yy}}$ also can be described analytically:
\begin{alignat*}{4}
  &\beta_{11} &&= 0   \qquad\qquad && \beta_{12} &&= 0 \\
  &\tau_1     &&= \pm \sqrt{c_{11}} \qquad\qquad && \tau_2 &&= \pm \sqrt{c_{22}}
\end{alignat*}
In summary, we have proved the following result.

\begin{theorem}
\label{thm:factoranalysis21} 
Suppose that $p=2$ and $q=1$. For a generic sample covariance matrix $C_{yy}$, the likelihood equations for the factor analysis model in Section \ref{factoranalysis} have an infinite number of real solutions.
\end{theorem}

For $p \geq 2$ and $q \geq 1$, we can apply the same approach as in the case $(p,q)=(2,1)$ to construct a sample covariance matrix $C_{yy}$ such that $V_{p,q, C_{yy}}$ is positive dimensional.

\begin{theorem}
\label{thm:factoranalysispq}
For all $p \geq 2$ and $q \geq 1$, there exists a covariance matrix $C_{yy}$ such that $V_{p,q, C_{yy}}$ is positive dimensional. 
\end{theorem}

\begin{proof}  Let $C_{yy}$ be a positive definite matrix of the following form:
$$
C_{yy} = \left(\begin{array}{ccccc}c_{11} & c_{12} & 0 & \cdots & 0 \\c_{12} & c_{22} & 0 & \cdots & 0 \\0 & 0 & c_{33} & \cdots & 0 \\\vdots & \vdots & \vdots & \ddots & \vdots \\0 & 0 & 0 & \cdots & c_{pp}\end{array}\right).
$$
The set of all $(\tau,\beta, R)$ such that $r_{kl}=0$ for $1 \leq k < l \leq q$ and
$$
C_{yy} = \tau^2+ \beta'\beta
$$
is a one-dimensional subset of $V_{p,q, C_{yy}}$.  Indeed, this set is obtained by solving for $\beta_{11}, \beta_{12}, \tau_1, \tau_2$ in equations \eqref{eq:onedim1}-\eqref{eq:onedim2} and then setting $\beta_{kl}=0$ for $(k,l) \neq (1,1), (1,2)$ and $\tau_k= \pm \sqrt{c_{kk}}$.
\end{proof}

\begin{remark}  
We note that the set of solutions in Theorem \ref{thm:factoranalysis21} is the set of solutions to the equation 
\begin{equation}
\label{eq: remark}
C_{yy} = \tau^2 + \beta' \beta.
\end{equation}
Since $f(\tau^2,\beta,R)$ is a function of $\tau^2 + \beta'\beta$ for the case in which $(p,q) = (2,1)$ then the likelihood function is constant on $V_{2,1,C_{yy}}$.  Thus, this set of solutions forms a ridge on the likelihood hypersurface.  Although it follows immediately from \eqref{eq: loglikelihood} that the log-likelihood function attains the same value at every solution to \eqref{eq: remark}, one cannot similarly conclude immediately that these solutions are critical points, the reason being that the first partial derivatives of $f(\tau^2,\beta,R)$ are not functions of $\tau^2 + \beta' \beta$.
\end{remark}

While Theorem~\ref{thm:factoranalysispq} describes some sample covariance matrices that result in an infinite number of solutions to the likelihood equations, we believe that there are more.  However, once we increase $p$ and $q$, even modestly, the computations become infeasible with current methodology.  One approach would be to use numerical algebraic geometry, but such approaches would require creative manipulation of the system.

\bigskip


\section*{Acknowledgments}


The research of Gross was supported by the National Science Foundation award DMS-1304167 and DMS-1620109.  The research of Richards was partially supported by the National Science Foundation under grant DMS-1309808.  The research of Stasi and Petrovi\'c was partially supported by Air Force Office for Scientific Research grant FA9550-14-1-0141. We thank Daniel Brake for helping us to visualize the curve in Theorem 4.1 using \verb|bertini_real| \cite{BertiniReal}.



\end{document}